\documentclass{article}
\usepackage{amssymb}
\usepackage{verbatim}
\usepackage{array}
\usepackage{latexsym}
\usepackage{enumerate}
\usepackage{amsmath}
\usepackage{amsthm}
\usepackage[english]{babel}
\usepackage{graphicx}
\usepackage[all]{xy}
\usepackage{boxedminipage}
\usepackage{psfrag}
\usepackage{epsfig}
\usepackage{float}
\usepackage{cite}
\usepackage[latin1]{inputenc}
\usepackage{cancel}

\newtheorem{theorem}{Theorem}[section]
\newtheorem{proposition}[theorem]{Proposition}
\newtheorem{lemma}[theorem]{Lemma}
\newtheorem{corollary}[theorem]{Corollary}
\newtheorem{remark}[theorem]{Remark}

\def\cC{\mathcal C}

\def\cH{\mathcal H}

\def\cX{\mathcal X}
\def\cY{\mathcal Y}

\def\K{\mathbb{F}_{q^{2n}}}

\def\ord{\mbox{\rm ord}}
\def\deg{\mbox{\rm deg}}

\newcommand{\aut}{\mbox{\rm Aut}}

\title{A new family of maximal curves}
\date{}
\author{Peter Beelen and Maria Montanucci}

\begin{document}
\maketitle

\begin{abstract}
In this article we construct for any prime power $q$ and odd $n \ge 5$, a new  $\mathbb{F}_{q^{2n}}$-maximal curve $\mathcal X_n$. Like the Garcia--G\" uneri--Stichtenoth maximal curves, our curves generalize the Giulietti--Korchm\'aros maximal curve, though in a different way. We compute the full automorphism group of $\mathcal X_n$, yielding that it has precisely $q(q^2-1)(q^n+1)$ automorphisms. Further, we show that unless $q=2$, the curve $\mathcal X_n$ is not a Galois subcover of the Hermitian curve. Finally, up to our knowledge, we find new values of the genus spectrum of $\mathbb{F}_{q^{2n}}$-maximal curves, by considering some Galois subcovers of $\mathcal X_n$.
\end{abstract}


\section{Introduction}

Let $\mathcal K$ be an absolutely irreducible, non-singular, projective algebraic curve defined over the finite field $\mathbb{F}_q$ with $q$ elements. The famous Hasse--Weil bound states that $\mathcal K$ can have at most $q+1+2g(\mathcal K)\sqrt{q}$ points defined over $\mathbb{F}_q$, where $g(\mathcal K)$ denotes the genus of the curve $\mathcal K$. The curve $\mathcal K$ is called $\mathbb{F}_q$-maximal if it attains the Hasse--Weil bound. A curve of genus $0$ is ``maximal'' for any $q$, but otherwise this can only be the case if the cardinality $q$ of the finite field is a square.

An important and well-studied example of an $\mathbb{F}_{q^{2}}$-maximal curve is given by the Hermitian curve $\mathcal{H}_q$.
It is a plane curve, which can be defined by the affine equation
$$\mathcal{H}_q: \ Y^{q+1}  = X^{q+1}-1.$$
Other equivalent affine equations are often given, such as the equation $Y^{q+1}=X^{q}+X,$ but for our purposes the above equation will be most convenient.
For fixed $q$, the curve $\mathcal{H}_q$ has the largest possible genus $g(\mathcal{H}_q) =q(q-1)/2$ that an $\mathbb{F}_{q^2}$-maximal curve can have. The automorphism group of $\mathcal H_q$ is known to be $PGU(3,q)$, which is a group of order $(q^3+1)q^3(q^2-1).$ A result commonly attributed to Serre, see \cite[Proposition 6]{L1987}, gives that any $\mathbb{F}_{q^2}$-rational curve which is covered by an $\mathbb{F}_{q^2}$-maximal curve is itself also $\mathbb{F}_{q^2}$-maximal. Therefore many maximal curves can be obtained by constructing subcovers of already known maximal curves, in particular subcovers of the Hermitian curve. For a while it was speculated in the research community that perhaps all maximal curves could be obtained as subcovers of the Hermitian curve, but it was shown by Giulietti and Korchm\'aros that this is not the case; see \cite{GK}. More precisely, they constructed an $\mathbb{F}_{q^6}$-maximal curve which is not a subcover of the Hermitian curve $\mathcal H_{q^3}$. We will denote this $\mathbb{F}_{q^6}$-maximal curve by $\mathcal C$ and call it the Giulietti--Korchm\'aros (GK) curve.
It is defined over $\mathbb{F}_{q^6}$ by the affine equations
$$\mathcal{C}: \begin{cases} Z^{q^2-q+1}  = Y \frac{X^{q^2}-X}{X^q+X}\\ Y^{q+1}  = X^q+X\end{cases} .$$
The GK curve can be seen as a cover of the Hermitian curve $\mathcal{H}_q.$

A generalization of $\mathcal{C}$, often called the Garcia--G\"uneri--Stichtenoth (GGS) curve or the generalized GK curve,  was introduced in \cite{GGS}. For any prime power $q$ and odd $n \ge 3$, they found an $\mathbb{F}_{q^{2n}}$-maximal curve $\mathcal C_n$ with genus $g(\mathcal C_n)=(q-1)(q^{n+1}+q^n-q^2)/2$, given by the affine equations
$$\mathcal{C}_n: \begin{cases} Z^{m}=Y \frac{X^{q^2}-X}{X^q+X}\\ Y^{q+1}=X^q+X\end{cases},$$
where $m:=(q^n+1)/(q+1)$. Note that in particular $\mathcal{C}=\mathcal{C}_3$. Moreover, it is known that for any $q$ and $n \ge 5$, the curve is not a Galois subcover of the Hermitian curve $\mathcal H_{q^n}$; \cite{DM,GMZ}.
The automorphism group of $\mathcal{C}$ was determined in \cite{GK}. There it was shown that each automorphism of $\mathcal H_q$ can be extended to an automorphism of $\mathcal C$ in exactly $m$ distinct ways and that each automorphism of $\mathcal C$ is obtained in that way. In other words: the automorphism group of $\mathcal H_q$, can be lifted completely and $\mathcal C$ has exactly $(q^3+1)q^3(q^2-1)m$ automorphisms. For $n \ge 5$ a similar statement is not true: it was shown in \cite{GOS,GMP} that only the stabilizer of the infinite point $Q_\infty$ of $\mathcal H_q$ can be lifted completely and that any automorphism of $\mathcal C_n$ is obtained in this way, yielding $|Aut(\mathcal{C}_n)|= q^3(q^2-1)m$. Note that from \cite[Theorem A.10]{HKT} the stabilizer of $Q_\infty$ is a maximal subgroup of $Aut(\mathcal H_q) \cong PGU(3,q)$ of order $q^3(q^2-1)$.

It is natural to ask whether or not it is possible to construct other generalizations of $\mathcal{C}$ such that other subgroups of $PGU(3,q)$ can be lifted completely. In this article was give a possible answer for the maximal subgroup of $PGU(3,q)$ given by the non-tangent line stabilizer of $\mathcal H_q$. This group has order $(q+1)q(q^2-1)$ and permutes the $q+1$ $\mathbb{F}_{q^2}$-rational points of $\mathcal H_q$ lying on the intersection of $\mathcal H_q$ with an $\mathbb{F}_{q^2}$-rational non-tangent line. In the following we construct for each odd $n\ge 5$, an $\mathbb{F}_{q^{2n}}$-maximal curve $\mathcal X_n$ in which precisely this non-tangent line stabilizer of $PGU(3,q)$ is lifted completely. Also we will show that any automorphism of $\mathcal X_n$ is obtained in this way, yielding that $\mathcal X_n$ has precisely $(q+1)q(q^2-1)m=q(q^2-1)(q^n+1)$ automorphisms. Moreover, we show that even though $g(\mathcal X_n)=g(\mathcal C_n)$, the curve $\mathcal X_n$ not isomorphic to $\mathcal C_n$ and cannot be Galois covered by the Hermitian curve $\mathcal H_{q^n}$ if $q>2$. Finally, we show that several new genera for $\mathbb{F}_{q^{2n}}$-maximal curves can be obtained by considering some of the subcovers of $\mathcal X_n$.

\section{The curves $\mathcal{X}_n$}

The starting point of our generalization of the GK curve is to consider a different model of $\mathcal{C}$. This model also occurs in \cite{GQZ}. Let $\rho \in \mathbb{F}_{q^2}$ with $\rho^q+\rho=1$, and consider the $\mathbb{F}_{q^2}$-projectivity $\varphi$ associated to the matrix $A$ where
$$A:=
\begin{pmatrix} 1 & 0 & 0 & 1\\ 0 & 1 & 0 & 0 \\ 0 & 0 & -1 & 0 \\ 1-\rho & 0 & 0 & -\rho \end{pmatrix}.$$
In terms of the affine coordinates $X, Y, Z$ this amounts to the algebraic transformation
$$\varphi(X)=\dfrac{X+1}{(1-\rho)X-\rho}, \ \varphi(Y)=\dfrac{Y}{(1-\rho)X-\rho}, \ \makebox{and}  \ \varphi(Z)=\dfrac{-Z}{(1-\rho)X-\rho}.$$
Also using the variables $X$, $Y$, and $Z$ for $\cX:=\varphi(\mathcal{C})$, we obtain that $\cX$ is given by the affine equations
$$\mathcal{X}: \begin{cases} Z^{q^2-q+1}=Y \frac{X^{q^2}-X}{X^{q+1}-1}\\ Y^{q+1}=X^{q+1}-1\end{cases} .$$
Inspired by this, we consider for each odd $n$ with $n \ge 3$, the algebraic curve $\mathcal X_n$ given by the affine equations
\begin{equation}\label{xn}
\mathcal{X}_n: \begin{cases} Z^{m}=Y \frac{X^{q^2}-X}{X^{q+1}-1}\\ Y^{q+1}=X^{q+1}-1\end{cases},
\end{equation}
where $m:=(q^n+1)/(q+1)$. In particular, as for the construction in \cite{GGS}, we note that $\mathcal{X}=\mathcal{X}_3,$ which we have seen to be isomorphic to $\mathcal{C}_3$. As we will see later, for $n \ge 5$ the curves $\mathcal X_n$ are new and in particular not isomorphic to $\mathcal C_n$. We start our investigation of the curves $\mathcal X_n$ by computing their genera.

\begin{proposition} \label{genus}
Let $q$ be a power of a prime $p$ and let $n \geq 5$ odd. Then the genus of $\mathcal{X}_n$ is given by
$$g(\mathcal{X}_n)=g(\mathcal{C}_n)=\frac{(q-1)(q^{n+1}+q^n-q^2)}{2}.$$
\end{proposition}

\begin{proof}
Let $\K(x,y)$ be the function field defined by $y^{q+1}=x^{q+1}-1$ and $\mathcal{H}_q: Y^{q+1}-X^{q+1}+1$ be the Hermitian curve over $\mathbb{F}_{q^2}$ as before. Seen as a projective curve, $\mathcal{H}_q$ is nonsingular having $q+1$ infinite points, namely $Q_{\infty}^i=(1:\alpha_i:0)$ where $\alpha_i^{q+1}=1$ for $i=1,\ldots,q+1$. Let $P_{\infty}^i$ with $i=1,\ldots,q+1$ be the place of $\K(x,y)$ corresponding to the point $Q_{\infty}^i$. Further, for $a \in \K$ such that $a^{q+1} \neq 1$, let $P_{x=a}^1,\ldots,P_{x=a}^{q+1}$ and $P_{y=0}^1,\ldots,P_{y=0}^{q+1}$ be the places lying over $x=a$ in $\K(x,y) / \K(x)$ and over $y=0$ in $\K(x,y)/\K(y)$ respectively. Then we have the following identities of divisors in $\K(x,y)$,
$$(x-a)_{\K(x,y)} = \sum_{j=1}^{q+1} P_{x=a}^j - \sum_{i=1}^{q+1} P_{\infty}^i, \quad  {\rm and}, \quad  (y)_{\K(x,y)}=\sum_{j=1}^{q+1} P_{y=0}^j - \sum_{i=1}^{q+1} P_{\infty}^i.$$
This implies that
\begin{align*}
\bigg( y \frac{x^{q^2}-x}{x^{q+1}-1} \bigg)_{\K(x,y)}
 &=\sum_{i=1}^{q+1} P_{y=0}^i-\sum_{j=1}^{q+1} P_{\infty}^j+ \sum_{\substack{a \in \mathbb{F}_{q^2}, \\ a^{q+1} \ne 1}} \ \sum_{k=1}^{q+1} P_{x=a}^k -(q^2-q-1) \sum_{j=1}^{q+1} P_{\infty}^{j}\\
&=\sum_{i=1}^{q+1} P_{y=0}^i+ \sum_{\substack{a \in \mathbb{F}_{q^2}, \\ a^{q+1} \ne 1}} \ \sum_{k=1}^{q+1} P_{x=a}^k -(q^2-q) \sum_{j=1}^{q+1} P_{\infty}^{j}.\\
\end{align*}
Thus, the support of the divisor of $y (x^{q^2}-x)/(x^{q+1}-1)$ consists of $q^3+1$ places of $\K(x,y)$, corresponding precisely to the $\mathbb{F}_{q^2}$-rational points of $\mathcal H_q$. The divisor of $y (x^{q^2}-x)/(x^{q+1}-1)$ also shows that this function is not a $d$-th power of another function in $\K(x,y)$ for any $d>1$. Hence $\K(x,y,z)/\K(x,y)$ is a Kummer extension of degree $m$. The standard theory of Kummer extensions, see for example \cite[Proposition 3.7.3]{Sti}, implies that the $q^3+1$ places occurring in the divisor of $y (x^{q^2}-x)/(x^{q+1}-1)$ are precisely the ramified places of $\K(x,y)$ in the extension $\K(x,y,z)/\K(x,y)$. Moreover all these places are totally ramified. Using the Riemann--Hurwitz formula and the fact that $g(\K(x,y))=q(q-1)/2$, yields
$$2g(\K(x,y,z))-2=m(q(q-1)-2)+(m-1)(q^3+1).$$
Hence for $n \geq 3$ odd, we obtain
$$g(\mathcal{X}_n)=g(\K(x,y,z))=\frac{(q-1)(q^{n+1}+q^n-q^2)}{2}=g(\mathcal{C}_n).$$
\end{proof}
The following corollary is a direct consequence of Proposition \ref{genus} and a result from \cite{DM}.

\begin{corollary} \label{notgal}
Let $n \ge 3$ be an odd integer and suppose $q \ge 3$. Then the curve $\mathcal{X}_n$ defined in Equation \eqref{xn}, is not a Galois subcover of the Hermitian curve $\mathcal H_{q^n}$.
\end{corollary}
\begin{proof}
In \cite{DM} it was shown that for $q \ge 3$ and $n \ge 3$ odd, the curve $\mathcal C_n$ is not a Galois subcover of the Hermitian curve $\mathcal H_{q^n}.$ There, it was noted in a remark directly following \cite[Proposition 5.1]{DM} that in fact for $q \geq 3$ a prime power and $n \geq 3$ an odd integer, \emph{any} curve defined over $\mathbb{F}_{q^{2n}}$ of genus $g(\mathcal C_n)=(q-1)(q^{n+1}+q^n-q^2)/2$ is not a Galois subcover of the Hermitian curve $\mathcal H_{q^n}.$ Hence the corollary follows from Proposition \ref{genus}.
\end{proof}

\begin{remark} \label{noteq}
Proposition \ref{genus} and the similarity between the definitions of $\mathcal C_n$ and $\mathcal X_n$ may give rise to the suspicion that these curves are isomorphic over $\mathbb{F}_{q^{2n}}$.
However, using the same projectivity $\varphi$ as before, note that $\tilde{\mathcal{C}}_n=\varphi({{\mathcal{C}}_n})$ is given by the affine equations:
$$\tilde{\mathcal{C}}_n: \begin{cases} Z^{m}=Y \frac{X^{q^2}-X}{X^{q+1}-1}(X-1)^{m-(q^2-q+1)}\\ Y^{q+1}=X^{q+1}-1\end{cases}.$$
This shows that $\mathcal{X}_3 \cong \mathcal{C}_3$, which we already knew, but no obvious isomorphism between $\mathcal{X}_n$ and $\mathcal{C}_n$ can be derived for $n\ge 5$. As mentioned already, we will show later that in fact the curves $\mathcal{X}_n$ and $\mathcal{C}_n$ are not isomorphic if $n \ge 5$.
\end{remark}

To complete the picture as to whether or not $\mathcal X_n$ is a Galois subcover of $\mathcal H_{q^n}$, we also settle the case $q=2$.
\begin{lemma} \label{q2}
Let $n \geq 3$ odd and $q=2$. Then $\mathcal{X}_n$ is a Galois subcover of the Hermitian curve $\mathcal{H}_{2^{n}}$.
\end{lemma}

\begin{proof}
For $q=2$, Equation \eqref{xn} reads:
$$\mathcal{X}_n: \begin{cases} Z^{m}=Y \frac{X^{4}+X}{X^{3}+1}=XY\\ Y^{3}=X^{3}+1\end{cases},$$
providing a plane model for $\mathcal{X}_n$ given by $\mathcal{Y}_n: Z^{2^n+1}=X^3(X^3+1).$  We claim that more generally for a prime power $r$, a plane curve $\mathcal{Y}$ given by an affine equation $\mathcal{Y}: Z^{r+1}=X^h(X^h+1)$ where $h \mid (r+1)$, is a Galois subcover of $\mathcal{H}_{r}: U^{r+1} - V^{r+1}+1=0$. To see this, consider the rational map $\psi: \mathcal{H}_{r} \rightarrow \mathcal{Y},$ with $\psi(U,V)=(X,Z)=(U^{\frac{r+1}{h}}, UV)$. Indeed, we note that $$Z^{r+1}-X^h(X^h+1)=(UV)^{r+1}-U^{2(r+1)}-U^{r+1}=U^{r+1}(V^{r+1} - U^{r+1}-1)=0.$$
This proves the claim.
Since $3$ divides $2^n+1$ for odd $n$, the lemma follows from the claim we just proved by applying it to the case $q=2^n$ and $h=3$.
\end{proof}

This lemma immediately implies that $\mathcal{X}_n$ is $\mathbb{F}_{2^{2n}}$-maximal, since it is a subcover of an $\mathbb{F}_{2^{2n}}$-maximal curve.
Lemma \ref{q2} also reveals a first difference between the curves $\cC_n$ and $\cX_n$. Indeed, for $q=2$, $\mathcal C_3$ was shown to be a Galois subcover of $\mathcal{H}_{q^3}$ in \cite{GK},  but in \cite{GMZ} it is proved that, for $q=2$ and $n \geq 5$, $\mathcal{X}_n$ is not a Galois subcover of $\mathcal{H}_{q^n}$.
Combining the above, we see that for $q=2$, the curves $\mathcal{C}_n$ and $\mathcal{X}_n$ are isomorphic if and only if $n =3$. We now prove that the same holds for $q \geq 3$. This result will follow as an immediate corollary of the following theorem.

\begin{theorem} \label{lift}
Let $q$ be a prime power and $n \geq 3$ odd. Then $Aut(\mathcal{X}_n)$ contains a subgroup $G$ with $G \cong SL(2,q) \rtimes C_{q^n+1}$, where $C_k$ denotes a cyclic group of order $k$. Moreover, denoting by $Fix(G)$, the fixed field of $G$ in $\K(x,y,z)$, we have $Fix(G)=\K(z^{q^n+1})$. Further, there exists a subgroup $S$ of $G$ with $S \cong SL(2,q) \times C_m$ and $Fix(S)=\K(z^m)=\K(y(x^{q^2}-x)/(x^{q+1}-1))$.
\end{theorem}

\begin{proof}
The Hermitian curve $\mathcal{H}_q$ defined by the affine equation $Y^{q+1}=X^{q+1}-1$ has, as we have noted already, $q+1$ points at infinity. These points lie on a non-tangent line $\ell$. From \cite{MZ} the stabilizer $M_\ell$ of a non-tangent line $\ell$ is a maximal subgroup of $Aut(\mathcal{H}_q)$ of order $q(q-1)(q+1)^2$. Moreover, $M_\ell$ is isomorphic to $S_\ell \rtimes C_{q+1}$ where $C_{q+1}$ is a cyclic group of order $q+1$, and $S_\ell$ is the commutator subgroup of $M_\ell$, which in turn is isomorphic to $SL(2,q)$. An explicit description of $M_\ell$ and $S_\ell$ is given in \cite{MZ} in matrix representation as follows:
$$M_\ell=\Bigg\{ \begin{pmatrix} a & b & 0 \\ c & d & 0 \\ 0 & 0 & 1\end{pmatrix} : a,b,c,d \in \mathbb{F}_{q^2}, \  d^{q+1}-b^{q+1}=1, \ a^{q+1}-c^{q+1}=1, -a^qb+c^qd=0 \Bigg \},$$
$$S_\ell=\Bigg\{ \begin{pmatrix} a & b & 0 \\ c & d & 0 \\ 0 & 0 & 1\end{pmatrix} \in M_\ell : ad-bc=1 \Bigg\} \cong SL(2,q),$$
also see \cite[Sect. 3]{CKT} for more details.

We will identify matrices from $M_\ell$ with their associated projectivities and denote with $\alpha_{a,b,c,d}$ the projectivity associated to the matrix in $M_\ell$ with entries $a,b,c,d.$
Then for any $\alpha_{a,b,c,d} \in M_\ell$, we have that $c^qd=a^qb$,
$$1+b^{q+1}=d^{q+1}=\frac{a^{q(q+1)}b^{q+1}}{c^{q(q+1)}}=\frac{(1+c^{q+1})^qb^{q+1}}{c^{q(q+1)}}=\bigg(\frac{b}{c^q}\bigg)^{q+1}+b^{q+1},$$
and
$$d^{q+1}-1=b^{q+1}=\frac{c^{q(q+1)}d^{q+1}}{a^{q(q+1)}}=\frac{(-1+a^{q+1})^qd^{q+1}}{a^{q(q+1)}}=-\bigg(\frac{d}{a^q}\bigg)^{q+1}+d^{q+1}.$$
This means that for any $\alpha_{a,b,c,d} \in M_\ell$ there exists $\zeta \in \mathbb{F}_{q^2}$ such that $\zeta^{q+1}=1$, $b=\zeta c^q$, and $d=\zeta a^q.$ In particular, this implies that $b=c^q$ and $a=d^q$ for every $\alpha_{a,b,c,d} \in S_\ell.$ Hence we have shown that
\begin{equation}\label{eq:Ml}
M_\ell=\Bigg\{ \begin{pmatrix} a & \zeta c^q & 0 \\ c & \zeta a^q & 0 \\ 0 & 0 & 1\end{pmatrix} : a,c \in \mathbb{F}_{q^2}, \  a^{q+1}-c^{q+1}=1, \zeta^{q+1}=1 \Bigg \}
\end{equation}
and
\begin{equation}\label{eq:Sl}
S_\ell=\Bigg\{ \begin{pmatrix} a & c^q & 0 \\ c & a^q & 0 \\ 0 & 0 & 1\end{pmatrix} : a,c \in \mathbb{F}_{q^2}, \  a^{q+1}-c^{q+1}=1 \Bigg\}.
\end{equation}
As before, denote by $\K(x,y)$ the Hermitian funcion field over $\K$ where $y^{q+1}-x^{q+1}+1=0$.
We are going to analyze the induced action on $\K(x,y)$ of $\alpha_{a,b,c,d} \in M_\ell$, where $b=\zeta c^q$ and $d=\zeta a^q$. First of all, note that a direct computation gives that for any $\alpha_{a,b,c,d} \in M_\ell$ we have
$$\alpha_{a,b,c,d}(y)^{q+1}-\alpha_{a,b,c,d}(x)^{q+1}+1=y^{q+1}-x^{q+1}+1.$$
Hence $\alpha_{a,b,c,d}$ is an automorphism of the function field $\K(x,y)$ (as well as of the curve $\mathcal H_q$).
Now consider the function $u:=y(x^{q^2}-x)/(x^{q+1}-1)=(x^{q^2}-x)/y^q$ in $\K(x,y)$. Then we have
$$\alpha_{a,b,c,d}(u)=(cx+dy) \bigg( \frac{ax^{q^2}+by^{q^2}-ax-by}{(cx+dy)^{q+1}}\bigg)=\bigg( \frac{ax^{q^2}+by^{q^2}-ax-by}{c^qx^q+d^qy^q}\bigg).$$
Observe that \begin{align*}
(x^{q^2}-x)(c^qx^q+d^qy^q)&=c^qx^{q^2+q}+d^qx^{q^2}y^q-c^qx^{q+1}-d^qy^qx\\
&=c^q(y^{q+1}+1)^q+d^qx^{q^2}y^q-c^q(y^{q+1}+1)-d^qy^qx\\
&=c^qy^{q^2+q}+d^qx^{q^2}y^q-c^qy^{q+1}-d^qy^qx\\
&=y^q(c^qy^{q^2}+d^qx^{q^2}-c^qy-d^qx).
\end{align*}
Combining this with Equation \eqref{eq:Ml}, we see that
$$\frac{\alpha_{a,b,c,d}(u)}{u}=\frac{y^q(ax^{q^2}+by^{q^2}-ax-by)}{(x^{q^2}-x)(c^qx^q+d^qy^q)}=\frac{ax^{q^2}+by^{q^2}-ax-by}{c^qy^{q^2}+d^qx^{q^2}-c^qy-d^qx}=\zeta,$$
where in the last equality we used that $b=\zeta c^q$ and $a=a^{q^2}=d^q/\zeta^q=\zeta d^q.$ In particular Equation \eqref{eq:Sl} now implies that $\alpha_{a,b,c,d}(u)=u$ if and only if $\alpha_{a,b,c,d} \in S_\ell$. Hence $u \in Fix(S_\ell)$ and $u^{q+1} \in Fix(M_\ell)$. On the other hand from the proof of Lemma \ref{genus}, we see that the pole divisor of $u$ has degree $q^3-q=|S_\ell|$. Hence $[\K(x,y):\K(u)]=|S_\ell|$ and $[\K(x,y):\K(u^{q+1})]=(q+1)|S_\ell|=|M_\ell|$. This implies that $\K(u)=Fix(S_\ell)$ and $\K(u^{q+1})=Fix(M_\ell)$.

Now let $\alpha_{a,b,c,d} \in M_\ell$, where $b=\zeta c^q$ and $d=\zeta a^q$. Since in $\K(x,y,z)$ we have that $z^m=u$ and $\alpha_{a,b,c,d}(u)=\zeta u$, for any $\alpha_{a,b,c,d} \in M_\ell$ and $\xi \in \K$ such that $\xi^m=\zeta$, the transformation $\alpha_{a,b,c,d,\xi}: (x,y,z) \mapsto (ax+by,cy+dy,\xi z)$ is an automorphism of $\K(x,y,z)$. Note that since $\zeta^{q+1}=1$, we have $\xi^{q^n+1}=1,$ implying that indeed $\xi \in \mathbb{F}_{q^{2n}}.$ Hence the groups of transformations
\begin{equation}\label{eq:groupG}
G:=\{\alpha_{a,\zeta c^q,c,\zeta a^q,\xi}: \alpha_{a,\zeta c^q,c,\zeta a^q} \in M_\ell, \ \xi^m=\zeta \}
\end{equation}
and
\begin{equation}\label{eq:groupS}
S:=\{\alpha_{a,c^q,c,a^q,\xi}: \alpha_{a,c^q,c,a^q} \in S_\ell, \ \xi^m=1 \}
\end{equation}
are automorphism groups of $\mathcal{X}_n$ with
$G \cong SL(2,q) \rtimes C_{q^n+1}$ and $S \cong S_\ell \times C_m \cong SL(2,q) \times C_m$. Further $Fix(G)=Fix(M_\ell)=\K(u^{q+1})=\K(z^{q^n+1})$ and $Fix(S)=Fix(S_\ell)=\K(u)=\K(z^m)$, since $[\K(x,y,z):K(z^{q^n+1})]=|M_\ell|m=|G|$ and $[\K(x,y,z):K(z^m)]=|S_\ell|m=|S|$.
\end{proof}

\begin{corollary} \label{notisom}
Let $q$ be a prime power and $n \geq 3$ odd. Then $\mathcal{X}_n$ is isomorphic to $\mathcal{C}_n$ if and only if $n =3$.
\end{corollary}
\begin{proof}
We already know that $\mathcal C_3$ and $\mathcal X_3$ are isomorphic. Now assume that $n \geq 5$. From Theorem \ref{lift}, we have that $q(q-1)(q+1)(q^n+1)$ divides $|Aut(\mathcal{X}_n)|$ while from \cite{GOS,GMP}, we know that $|\aut(\mathcal{C}_n)|=q^3(q-1)(q^n+1)=q^3(q-1)(q^n+1)$. Since $q(q-1)(q+1)(q^n+1)$  does not divide $q^3(q-1)(q^n+1)$ we conclude that $\mathcal{X}_n$ is not isomorphic to $\mathcal{C}_n$.
\end{proof}

\begin{remark}\label{rem:Galoisgroups}
The proof of Theorem \ref{lift} directly implies several facts on certain function field extensions and their corresponding Galois groups: the field extension $\K(x,y,z)/\K(z^{q^n+1})$ is Galois with Galois group isomorphic to the group $G$ given in Equation \eqref{eq:groupG}. Similarly, the extension $\K(x,y,z)/\K(z^{m})$ is Galois with Galois group isomorphic to the group $S$ given in Equation \eqref{eq:groupS}. In the next section, we will also use the following subgroup of $S$:
\begin{equation}\label{eq:groupSt}
\tilde{S}:=\{\alpha_{a,c^q,c,a^q,1}: \alpha_{a,c^q,c,a^q} \in S_\ell\},
\end{equation}
 which is directly seen to be the Galois group of the Galois extension $\K(x,y,z)/\K(z)$. Indeed any element from $\tilde{S}$ fixes $z$, while $\tilde{S}$ has cardinality $|SL(2,q)|=q^3-q=[\K(x,y,z):\K(z)]$, since it is isomorphic to the group $S_\ell$ from Equation \eqref{eq:Sl}, which in turn was isomorphic to $SL(2,q).$  Finally note that the Theorem \ref{lift} directly implies that group $S_\ell$ is the Galois group of the Galois extension $\K(x,y)/\K(u)$, with $u=y(x^{q^2}-x)/(x^{q+1}-1).$
\end{remark}

\section{$\mathcal{X}_n$ is $\mathbb{F}_{q^{2n}}$-maximal for every $n \geq 5$ odd}

In this section we prove that the algebraic curves $\mathcal X_n$ are $\K$-maximal for any $q$ and $n \ge 3$ odd. The proof proceeds in several steps, proving the maximality of certain key subfields of $\K(x,y,z)$ first, before showing the maximality of $\K(x,y,z)$ itself. An overview of the various function fields is given in Figure \ref{fig:one}. Some of the stated extension degrees and equations may not be clear at this point, but will be explained later in this section.
\begin{figure}[h!]
\begin{center}
\resizebox{\textwidth}{!}{
\makebox[\width][c]{
\def\objectstyle{\scriptstyle}
\xymatrix@!=1.45pc@r{
&&&& \K(x,y,z)&&z^m=y\frac{x^{q^2}-x}{x^{q+1}-1}\\
y^{q+1}=x^{q+1}-1&& \K(x,y)\ar@{-}[rru]^{m} \\
\left(\frac{x}{\rho}\right)^q+\left(\frac{x}{\rho}\right)=\frac{\rho^{q+1}+1}{\rho^{q+1}}&&&& \K(\rho,z)\ar@{-}_{q}[uu]&&z^m=\frac{\rho^{q^2-1}-1}{\rho^{q-1}}\\
\rho=x+y & & \K(\rho)\ar@{-}[uu]^{q}\ar@{-}[rru]^{m}&&&&\rho^{q-1}=\frac{1}{z^m(\eta-1)}=\frac{\delta-1}{z^m}\\
\delta=\rho^{q^2-1}&&&&\K(\eta,z)\ar@{-}[uu]_{q-1}&&\left(\frac{1}{z}\right)^{q^n+1}=\eta^{q+1}-\eta\\
\eta=\frac{\delta}{\delta-1}&&\K(\delta)=\K(\eta)\ar@{-}^{q^2-1}[uu]\ar@{-}^{q^n+1}[rru]
}
}
}
\end{center}
\caption{An overview of several subfields of $\K(x,y,z)$.\label{fig:one}}
\end{figure}
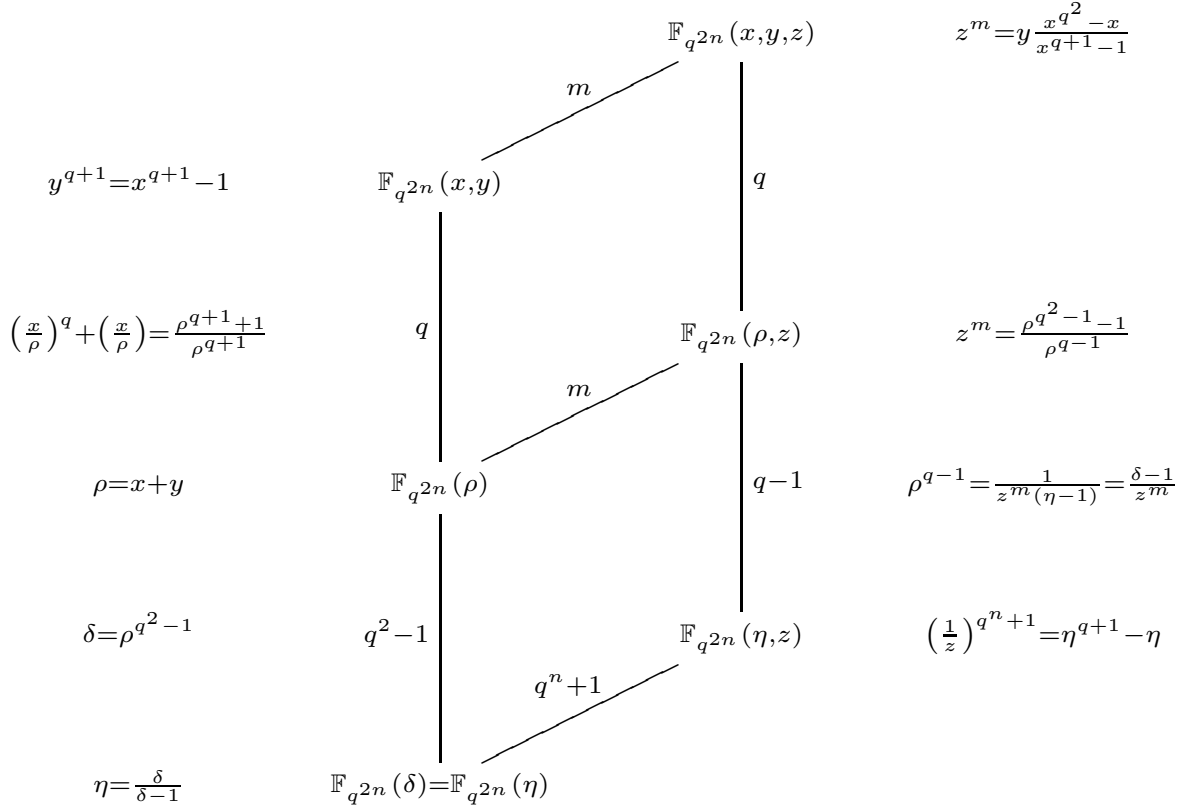

\begin{lemma}
Let $\K(x,y)$ be as before defined by $y^{q+1}-x^{q+1}+1=0$ and define $\rho:=x+y,$ $\delta:=\rho^{q^2-1}.$ Then $\K(x,y)/\K(\rho)$ is an Artin--Schreier extension of degree $q$ and $\K(\rho)/\K(\delta)$ is a Kummer extension of degree $q^2-1$.
\end{lemma}
\begin{proof}
The statement that $\K(\rho)/\K(\delta)$ is a Kummer extension of degree $q^2-1$ is trivial and we therefore only prove the first statement. As in the previous section denote by $P_{\infty}^i$, $i=1,\ldots,q+1$, the places of $\K(x,y)$ corresponding to the ``infinite'' points $Q_{\infty}^i=(1:\alpha_i:0)$ on the algebraic curve $Y^{q+1}-X^{q+1}+1=0.$ We make the convention that $Q_\infty^1=(1:-1:0)$. Since $Q_{\infty}^1$ is a point of the polar line of $(0:0:1)$ at $\mathcal{H}_q$, the Reciprocity Law of polarities yields that $(0:0:1)$ is a point of the polar line of $Q_{\infty}^1$ which is the tangent line of $\mathcal{H}_q$ at $Q_{\infty}^1$. Thus, the line $X+Y=0$ intersects $\mathcal{H}_q$ at $Q_{\infty}^1$ with multiplicity $q+1$. Hence
\begin{equation}\label{div:rho}
(\rho)_{\K(x,y)}=(x+y)_{\K(x,y)}=(q+1)P_\infty^1-\sum_{i=1}^{q+1}P_\infty^i=qP_\infty^1-\sum_{i=2}^{q+1}P_\infty^i.
\end{equation}
This shows that $[\K(x,y):\K(\rho)]=q.$
To find an algebraic relation between $\rho$ and $x$, note that $\rho^{q+1}=x^{q+1}+y^{q+1}+x^qy+xy^q=2x^{q+1}-1+x^qy+xy^q=x^q \rho + x \rho^q-1$, and hence
\begin{equation} \label{eqpol}
\bigg(\frac{x}{\rho}\bigg)^q + \bigg(\frac{x}{\rho} \bigg)=\frac{\rho^{q+1}+1}{\rho^{q+1}}.
\end{equation}
This implies that the extension $\K(x,y)/\K(\rho)$ is an Artin--Schreier extension.
\end{proof}

\begin{remark}
It is not hard to describe the Galois group of the extension $\K(x,y)/\K(\rho)$ as a subgroup of the group $M_\ell$ from Equation \eqref{eq:Ml}. Indeed, from Equation \eqref{eqpol} we see that an automorphisms $\sigma$ of $\K(x,y)/\K(\rho)$ sends $x/\rho$ to $x/\rho+\alpha$ with $\alpha$ satisfying $\alpha^q+\alpha=0$. In terms of $x$ and $y$, we obtain $$\sigma(x)=\rho\cdot \left(\frac{x}{\rho}+\alpha\right)=(\alpha+1)x+\alpha y \ \makebox{ and } \ \sigma(y)=\rho-\sigma(x)=-\alpha x+(1-\alpha)y.$$
Hence the group of such automorphisms $\sigma$ can be identified with the Sylow $p$-subgroup of $S_\ell$ given by
$$S_{\ell,q}=\Bigg\{ \begin{pmatrix} 1+\alpha & \alpha & 0 \\ -\alpha & 1-\alpha & 0 \\ 0 & 0 & 1\end{pmatrix} : \alpha^q+\alpha=0 \Bigg\}.$$
\end{remark}

We now turn our attention to the function field $\K(\eta,z)$ from Figure \ref{fig:one}, where $\eta:=\delta/(\delta-1)$. Note that $\delta=\eta/(\eta-1)$, implying that $\K(\eta)=\K(\delta).$ First we determine the genus of $\K(\eta,z)$.
\begin{lemma}\label{lem:genusrhoz}
Let $\delta:=\rho^{q^2-1}$ and $\eta:=\delta/(\delta-1)$. Then the extension $\K(\eta,z)/\K(\eta)$ is a Kummer extension of degree $q^n+1$. The function field $\K(\eta,z)$ has genus $g(\K(\eta,z))=\frac12(m-1)(q+1).$
\end{lemma}
\begin{proof}
We will start the proof by finding an algebraic relation between $z$ and $\delta$. First of all, note that since $\rho=x+y$ and $y^{q+1}=x^{q+1}-1$, we have that
\begin{align*}
y^{q+1}(\rho^{q^2}-\rho)& = y^{q+1}(x^{q^2}-x)+(x^{q+1}-1)(y^{q^2}-y)\\
 &=y^{q+1}(x^{q^2}-x)+(x^{q+1}-1)y((y^{q+1})^{q-1}-1)\\
 &=y^{q+1}(x^{q^2}-x)+(x^{q+1}-1)y((x^{q+1}-1)^{q-1}-1)\\
 &=y^{q+1}(x^{q^2}-x)+y(x^{q^2+q}-x^{q+1})=\rho^q y (x^{q^2}-x).\\
\end{align*}
Hence
\begin{equation}\label{eq:zrho}
z^m=\frac{y(x^{q^2}-x)}{x^{q+1}-1}=\frac{x^{q^2}-x}{y^q}=\frac{\rho^{q^2}-\rho}{\rho^q}=\frac{\rho^{q^2-1}-1}{\rho^{q-1}}.
\end{equation}
Since $\delta=\rho^{q^2-1},$ taking the $q+1$-th power on both sides of the above equation shows that
\begin{equation*}
z^{q^n+1}=\frac{(\delta-1)^{q+1}}{\delta}.
\end{equation*}
Hence in terms of $\eta=\delta/(\delta-1)$, we obtain that
\begin{equation}\label{eq:etaz}
\left(\frac{1}{z}\right)^{q^n+1}=\frac{\delta}{(\delta-1)^{q+1}}=\eta^{q+1}-\eta.
\end{equation}
This shows that $\K(\eta,z)/\K(\eta)$ is a Kummer extension of degree $q^n+1$.
The genus now follows directly from the theory of Kummer extensions: the two zeroes of $\eta^{q+1}-\eta=\eta(\eta-1)^q$ are totally ramified, while there are $q+1$ places of $\K(\eta,z)$ above the pole of $\eta^{q+1}-\eta$, each with ramification index $m$, since the pole order $q+1$ divides the extension degree $q^n+1$. Applying the Riemann-Hurwitz genus formula to the extension $\K(\eta,z)/\K(\eta)$ we obtain
$$2g(\K(\eta,z))-2=(q^n+1)(-2)+2q^n+(q+1)(m-1).$$ The desired result now follows directly.
\end{proof}

We now prove maximality of the function field $\K(\eta,z)$. The essential ingredient of the proof is to use a Galois-descent-type argument by interpreting $\K$ as a $2$-dimensional space over $\mathbb{F}_{q^n}.$

\begin{proposition} \label{etazismax}
Let $q$ be a prime power and $n \geq 3$ odd. The function field $\K(\eta,z)$ where $z^{q^n+1}=\eta^{q+1}-\eta,$ is $\mathbb{F}_{q^{2n}}$-maximal.
\end{proposition}
\begin{proof}
Using Lemma \ref{lem:genusrhoz}, the theorem is equivalent with the statement that the function field $\K(\eta,z)$ has precisely
$$q^{2n}+1+(q+1)(m-1)q^n=q+3+(q^n+1)\left( q^n-2 + q^n-q \right)$$ many $\K$-rational places. Above the poles and zeroes of $\eta^{q+1}-\eta$ in $\K(\eta)$ lie exactly $2+q+1$ $\K$-rational places of $\K(\eta,z)$. Further, for any $\eta \in \mathbb{F}_{q^n}\setminus \{0,1\}$, the equation $z^{q^n+1}=\eta^{q+1}-\eta$ has $q^n+1$ distinct solutions in $\K$, since then $\eta^{q+1}-\eta \in \mathbb{F}_{q^n} \setminus \{0\}$. This gives rise to precisely $q^n-2$ places of $\K(\eta)$ that split completely in the extension $\K(\eta,z)/\K(\eta).$ What remains to be shown is that there exist $q^n-q$ values of $\eta \in \K \setminus \mathbb{F}_{q^n}$ such that the equation
$z^{q^n+1}=\eta^{q+1}-\eta$ has $q^n+1$ distinct solutions in $\K$. Equivalently, we need to shown that there are $q^n-q$ values of $\eta \in \K \setminus \mathbb{F}_{q^n}$ such that $\eta^{q+1}-\eta \in \mathbb{F}_{q^n} \setminus \{0\}.$

Now fix $\alpha \in \mathbb{F}_{q^2} \setminus \mathbb{F}_q$ and suppose that $\eta \in \K \setminus \mathbb{F}_{q^n}$. Since $n$ is odd, any element $\eta \in \mathbb{F}_{q^{2n}} \setminus \mathbb{F}_{q^{n}}$ can uniquely be written in the form $\eta=\eta_1+\alpha \eta_2,$ with $\eta_1,\eta_2\in \mathbb{F}_{q^n}$ and $\eta_2 \neq 0.$ Now note that
\begin{align*}
\eta^{q+1}-\eta & = (\eta_1+\alpha \eta_2)^{q+1}-(\eta_1+\alpha \eta_2)\\
 & =  \eta_1^{q+1}+\eta_1^q\alpha \eta_2+\eta_1\alpha^q\eta_2^q+\alpha^{q+1}\eta_2^{q+1}-\eta_1-\alpha \eta_2\\
 & = \eta_1^{q+1}+\eta_1(\alpha^q+\alpha)\eta_2^q+\alpha^{q+1}\eta_2^{q+1}-\eta_1+\alpha(\eta_1^q\eta_2-\eta_1\eta_2^q-\eta_2).
\end{align*}
This implies that $\eta^{q+1}-\eta \in \mathbb{F}_{q^n}\setminus \{0\}$ if and only if $\eta_1^q\eta_2-\eta_1\eta_2^q-\eta_2=0.$
Since we require that $\eta_2 \neq 0$, we obtain that
\begin{equation}\label{eq:N}
\eta_1^q-\eta_1\eta_2^{q-1}-1=0 \ \makebox{ or equivalently } \ \eta_2^{q-1}=\frac{(\eta_1-1)^q}{\eta_1}.
\end{equation}
Now we compute the number of possibilities for $(\eta_1,\eta_2)$ by considering above equation as a defining equation of a function field $\mathbb{F}_{q^n}(\eta_1,\eta_2).$ By Equation \eqref{eq:N}, the extension $\mathbb{F}_{q^n}(\eta_1,\eta_2)/\mathbb{F}_{q^n}(\eta_1)$ is a Kummer extension of degree $q-1$ in which the zeroes of $\eta_1$ and $\eta_1-1$ in $\mathbb{F}_{q^n}(\eta_1)$ are totally ramified with ramification index $q-1$, while the pole of $\eta_1$ in $\mathbb{F}_{q^n}(\eta_1)$ splits completely. This implies that the genus of $\mathbb{F}_{q^n}(\eta_1,\eta_2)$ is zero and hence that it is has $q^n+1$ many $\mathbb{F}_{q^n}$-rational places. Of these, exactly $q+1$ are either a pole of $\eta_2$ (equivalently a pole or zero of $\eta_1$) or a zero of $\eta_1-1.$ This leaves $q^n-q$ many $\mathbb{F}_{q^n}$-rational places of $\mathbb{F}_{q^n}(\eta_1,\eta_2)$. Each of these places give rise to a distinct value of $\eta \in \mathbb{F}_{q^{2n}}$ satisfying that $\eta_2 \neq 0$, while $\eta^{q+1}-\eta \in \mathbb{F}_{q^n}^*.$ This is what we needed to show.
\end{proof}

\begin{corollary}\label{cor:splitting}
For a given $\alpha \in \mathbb{F}_{q^2} \setminus \mathbb{F}_q$, a place $P$ of $\K(\delta)$ splits completely in the extension $\K(\eta,z)/\K(\eta)$ if and only the evaluation $\delta(P)$ of $\delta$ at $P$ satisfies:
\begin{enumerate}
\item $\delta(P) \in \mathbb{F}_{q^n} \setminus \{0,1\}$, or
\item $\delta(P)=\dfrac{1+\alpha t^q}{t^{q-1}(1+\alpha t)},$ for $t \in \mathbb{F}_{q^n} \setminus \mathbb{F}_q.$
\end{enumerate}
\end{corollary}
\begin{proof}
From the proof of Proposition \ref{etazismax} we see that a splitting place $P$ comes about in two ways.

If $\eta(P)$ is in $\mathbb{F}_{q^n} \setminus \{0,1\}$, then, since $\delta=\eta/(\eta-1)$, we have $\delta(P)\in \mathbb{F}_{q^n} \setminus \{0,1\}.$

If $\eta(P)=\eta_1+\alpha \eta_2$ where $\eta_1,\eta_2 \in \mathbb{F}_{q^n}$, $\eta_2 \neq 0$ and $\eta_1^q-\eta_1\eta_2^{q-1}-1=0,$ then we define $t:=\frac{\eta_2}{\eta_1-1}.$ Using Equation \eqref{eq:N}, we can express $\eta_1$ and $\eta_2$ in terms of $t$ yielding that $$\eta(P)=\eta_1+\alpha \eta_2=\frac{1}{1-t^{q-1}}+\alpha\frac{t^q}{1-t^{q-1}}=\frac{1+\alpha t^q}{1-t^{q-1}}.$$
Hence
$$\delta(P)=\frac{\eta(P)}{\eta(P)-1}=\frac{1+\alpha t^q}{1-t^{q-1}} \cdot \frac{1-t^{q-1}}{1+\alpha t^q+t^{q-1}-1}=\frac{1+\alpha t^q}{t^{q-1}(1+\alpha t)}.$$
Recall from the proof of Proposition \ref{etazismax} that the zero and pole of $\eta_1$ as well as the zero of $\eta_1-1$ should not be considered when studying completely splitting places. Since $\eta_1=1/(1-t^{q-1})$, this means that $t \in \mathbb{F}_{q^n}$, but $t \not\in \mathbb{F}_q$ just as claimed.
\end{proof}
To prove the maximality of $\mathcal X_n$, or equivalently of $\K(x,y,z)$, the strategy is now to study which of the places in Corollary \ref{cor:splitting} split completely in the extension $\K(x,y)/\K(\delta)$ as well. Indeed, in light of Figure \ref{fig:one}, such places will also split completely in the extension $\K(x,y,z)/\K(\delta),$ since it is the composite of $\K(x,y)$ and $\K(\delta,z).$ We will first consider the extension $\K(\rho)/\K(\delta).$

The description of the splitting places in the previous corollary seems to depend on the choice of $\alpha \in \mathbb{F}_{q^2} \setminus \mathbb{F}_q.$ However, note that the set of splitting places as a whole does not depend on $\alpha.$ The same is true for the set of splitting places in the next proposition.

\begin{proposition}\label{prop:splittingtorho}
For a given $\alpha \in \mathbb{F}_{q^2} \setminus \mathbb{F}_q$, let $P$ be a place of $\K(\delta)$ such that the evaluation $\delta(P)$ of $\delta$ at $P$ satisfies:
\begin{enumerate}
\item $\delta(P) \in (\mathbb{F}_{q^n})^{q-1}\setminus \{0,1\}$, or
\item $\delta(P)=\dfrac{1+\alpha t^q}{t^{q-1}(1+\alpha t)},$ for $t \in \mathbb{F}_{q^n} \setminus \mathbb{F}_q.$
\end{enumerate}
Then the place $P$ splits completely in the function field extension $\K(\rho)/\K(\delta).$
\end{proposition}
\begin{proof}
Since $\delta=\rho^{q^2-1}$, it is clear that the equation $\rho^{q^2-1}=\delta(P)$ has $q^2-1$ distinct solutions in $\K$ if $\delta(P)$ is a $(q-1)$-th power in $\mathbb{F}_{q^n}.$ Since $\delta(P) \not \in \{0,1\}$ by Corollary \ref{cor:splitting}, the first part of the proposition follows.

Now suppose that $\delta(P)=\dfrac{1+\alpha t^q}{t^{q-1}(1+\alpha t)},$ for $t \in \mathbb{F}_{q^n} \setminus \mathbb{F}_q.$ To see that $\delta(P)$ is a $(q^2-1)$-th power in $\mathbb{F}_{q^n}$, observe that
since $t \in \mathbb{F}_{q^n}$, then $t^{q-1}$ is a $(q^2-1)$-th power in $\mathbb{F}_{q^{2n}}$. Hence, $\delta$ is a $(q^2-1)$-th power in $\mathbb{F}_{q^{2n}}$ if and only if $(1+\alpha t^q)/(1+\alpha t)$ is. Since $t^{q^{2n+1}} =t^q$ and $\alpha^{q^2}=\alpha$ we have that
$$(1+\alpha t^q)=(1+\alpha t^{q^{2n-1}})^{q^2}=(1+\alpha t^{q^{2n-1}})^{q^2-1} (1+\alpha t^{q^{2n-1}}).$$
Similarly,
$$(1+\alpha t^{q^{2n-1}})=(1+\alpha t^{q^{2n-3}})^{q^2}=(1+\alpha t^{q^{2n-3}})^{q^2-1} (1+\alpha t^{q^{2n-3}})$$
and proceeding iteratively, we obtain
$$(1+\alpha t^{2n-(n-2)})=(1+\alpha t^{q^{n}})^{q^2}=(1+\alpha t^{q^{n}})^{q^2-1} (1+\alpha t^{q^{n}})=(1+\alpha t^{q^{n}})^{q^2-1} (1+\alpha t).$$
This shows that
\begin{equation}\label{eq:delta}
\delta=\frac{1}{t^{q-1}} \cdot \bigg( \prod_{\substack{1 \le i \le n \\ i \ odd}} (1+\alpha t^{q^{2n-i}})\bigg)^{q^2-1}  \cdot \frac{1+\alpha t}{1+\alpha t}=\frac{1}{t^{q-1}} \cdot \bigg( \prod_{\substack{1 \le i \le n \\ i \ odd}} (1+\alpha t^{q^{n-i}})\bigg)^{q^2-1},
\end{equation}
which implies the desired.
\end{proof}

\begin{corollary}\label{ourABQmax}
The function field $\K(\rho,z)$ given by the equation $z^m=\frac{\rho^{q^2-1}-1}{\rho^{q-1}}$ is an $\K$-maximal function field of genus $g(\K(\rho,z))=\frac12(m-1)(q^2-1).$
\end{corollary}
\begin{proof}
The extension $\K(\rho,z)/\K(\rho)$ is a Kummer extension of degree $m$ in which precisely $q^2+1$ places ramify, namely the $q^2$ zeroes of $\rho^{q^2}-\rho$ and the pole of $\rho$. Moreover, all these places are totally ramified and $\K$-rational. Hence by the Riemann--Hurwitz genus formula, we obtain the genus of $\K(\rho,z)$ directly.

To prove the maximality, we need to show that apart from the $q^2$ zeroes of $\rho^{q^2}-\rho$ and the pole of $\rho$ in $\K(\rho,z)$, the function field $\K(\rho,z)$ contains a further
$$q^{2n}+1+q^n(m-1)(q^2-1)-q^2-1=m (q^2-1)\left( \frac{q^n-q}{q-1}+q^n-q\right)$$
$\K$-rational places. However, combining Corollary \ref{cor:splitting} and Proposition \ref{prop:splittingtorho}, we see that exactly $(q^n-q)/(q-1)+q^n-q$ places of $\K(\delta)$ split completely in the extension $\K(\rho,z)/\K(\delta),$ since $\K(\rho,z)$ is the composite of $\K(\rho)$ and $\K(\eta,z);$ see Figure \ref{fig:one}.
\end{proof}

The following corollary is a direct consequence of the proofs of the previous results and among others describes the splitting places in the extension $\K(\rho,z)/\K(\rho)$. It will be useful to prove the maximality of $\K(x,y,z)$.

\begin{corollary}\label{cor:splitting2}
Let $\alpha \in \mathbb{F}_{q^2}\setminus \mathbb{F}_q$ be given. The $\K$-rational places of $\K(\rho,z)$ lie above the following three types of places of $\K(\rho)$:
\begin{enumerate}
\item The $q^2+1$ totally ramified places of $\K(\rho)$ in $\K(\rho,z) / \K(\rho)$. More specifically, these are the $q^2$ places $P$ in $\K(\rho)$ such that $\rho^{q+1}(P) \in \mathbb{F}_q$ together with the pole of $\rho$ in $\K(\rho).$
\item The $(q+1)(q^n-q)$ places $P$ of $\K(\rho)$ satisfying that $\rho^{q+1}(P) \in \mathbb{F}_{q^{n}} \setminus \mathbb{F}_{q}.$ These places split completely in $\K(\rho,z) / \K(\rho)$.
\item The $(q^2-1)(q^n-q)$ places $P$ of $\K(\rho)$ satisfying that
$$\rho^{q+1}(P)=\frac{a}{t} \cdot \bigg( \prod_{j=0}^{n} (1+\alpha^{q^{j+1}} t^{q^{n-j}})\bigg),$$ for some $a \in \mathbb{F}_q\setminus \{0\}$ and $t \in \mathbb{F}_{q^n} \setminus \mathbb{F}_q$. These places split completely in $\K(\rho,z) / \K(\rho)$.
\end{enumerate}
\end{corollary}
\begin{proof}
Apart from item 3., this follows directly from Proposition \ref{prop:splittingtorho} and Corollary \ref{ourABQmax}, since $\delta=\rho^{q^2-1}$. To derive the expression in the third type of places, observe that Equation \eqref{eq:delta} implies that
$$\rho^{q+1}(P)=\frac{a}{t} \cdot \bigg( \prod_{\substack{1 \le i \le n \\ i \ odd}} (1+\alpha t^{q^{n-i}})\bigg)^{q+1}=\frac{a}{t} \cdot \bigg( \prod_{j=0}^{n} (1+\alpha^{q^{j+1}} t^{q^{n-j}})\bigg),$$
with $a \in \mathbb{F}_q \setminus\{0\}.$ In the second equality, we used that $\alpha \in \mathbb{F}_{q^2}$ and hence $\alpha^{q^2}=\alpha$ as well as that $t \in \mathbb{F}_{q^n}.$
\end{proof}


Now we prove the main result of this section.

\begin{theorem}\label{thm:maximal}
Let $q$ be a prime power and $n \ge 3$ an odd integer. The curve $\mathcal X_n$ with function field $\K(x,y,z)$ defined by $y^{q+1}=x^{q+1}-1$ and $z^m=y(x^{q^2}-x)/(x^{q+1}-1)$, is an $\K$-maximal curve.
\end{theorem}
\begin{proof}
By Proposition \ref{genus}, the maximality of $\mathcal X_n$ follows, if we can show that $\K(x,y,z)$ has precisely the following number of $\K$-rational places:
\begin{equation}\label{eq:helpxyz}
q^{2n}+1+q^n(q-1)(q^{n+1}+q^n-q^2)=1+q^3+1+mq\left( (q+1)(q^n-q)+(q^2-1)(q^n-q)\right).
\end{equation}
It is easy to see that the $q^2+1$ places of type 1. from Corollary \ref{cor:splitting2} gives rise to exactly $q^3+1$ $\K$-rational places of $\K(x,y)$ (indeed, these places correspond exactly to the $q^3+1$ $\mathbb{F}_{q^2}$-rational points of the Hermitian curve $\mathcal H_q$). All these $q^3+1$ places are totally ramified in the extension $\K(x,y,z)/\K(x,y)$. Hence maximality of $\K(x,y,z)$ follows if we can show that the remaining $(q+1)(q^n-q)+(q^2-1)(q^n-q)$ places of $\K(\rho)$ from Corollary \ref{cor:splitting2}, split completely in the extension $\K(x,y,z)/\K(\rho)$; also see Figure \ref{fig:one}.

Recall Equation \eqref{eqpol} describing the extension $\K(x,y)/\K(\rho)$ explicitly as an Artin--Schreier extension:
$$\bigg(\frac{x}{\rho}\bigg)^q + \bigg(\frac{x}{\rho} \bigg)=\frac{\rho^{q+1}+1}{\rho^{q+1}};$$
also see Figure \ref{fig:one}. If $P$ is one of the $(q+1)(q^n-q)$ places of $\K(\rho)$ satisfying that $\rho^{q+1}(P) \in \mathbb{F}_{q^{n}} \setminus \mathbb{F}_{q},$ then $(\rho^{q+1}(P)+1)/\rho^{q+1}(P) \in \mathbb{F}_{q^n}.$ By Equation \eqref{eqpol}, we see that $P$ splits completely in the extension $\K(x,y)/\K(\rho)$. Since by Corollary \ref{cor:splitting2}, the place $P$ also splits completely in $\K(\rho,z)/\K(\rho)$, we conclude that $P$ also splits completely in $\K(x,y,z)/\K(\rho)$, since $\K(x,y,z)$ is the compositum of $\K(\rho,z)$ and $\K(x,y)$.

Now suppose that $P$ is one of the $(q^2-1)(q^n-q)$ places of $\K(\rho)$ satisfying that $$\rho^{q+1}(P)=\frac{a}{t} \cdot \bigg( \prod_{j=0}^{n} (1+\alpha^{q^{j+1}} t^{q^{n-j}})\bigg),$$ for some $a \in \mathbb{F}_q\setminus \{0\}$ and $t \in \mathbb{F}_{q^n} \setminus \mathbb{F}_q$. The theorem is proved, if we can show the claim that these places all split completely in the extension $\K(x,y)/\K(\rho)$. Since
$$T^{q^{2n}}-T=\sum_{\ell=0}^{2n-1}(-1)^\ell (T^q+T)^{q^\ell},$$
the equation $T^q+T=\frac{\rho^{q+1}(P)+1}{\rho^{q+1}(P)}$ has $q$ solutions in $\K$ if and only if
$$
\sum_{\ell=0}^{2n-1}(-1)^\ell \left(\frac{\rho^{q+1}(P)+1}{\rho^{q+1}(P)}\right)^{q^\ell}=0,
$$
or equivalently, if and only if
\begin{equation}\label{eq:conditionsplit}
\sum_{\ell=0}^{2n-1}(-1)^\ell \left(\frac{1}{\rho^{q+1}(P)}\right)^{q^\ell}=0.
\end{equation}
Let $A:=\prod_{j=0}^{n-1} (1+\alpha^{q^j} t^{q^{n-j}})$, so that $\rho^{q+1}(P)=\frac{a}{t} \cdot {A^{q^n}}\cdot (1+\alpha t) $ and hence $\frac{1}{\rho^{q+1}(P)}=b \cdot \frac{A \cdot t}{1+\alpha t}$ for some $b \in \mathbb{F}_q \setminus \{0\}$. Here we used that $A^{q^n+1}=\prod_{j=0}^{2n-1}(1+\alpha^{q^j}t^{2n-j}) \in \mathbb{F}_q \setminus \{0\},$ since $$\left(A^{q^n+1}\right)^q=\prod_{j=0}^{2n-1}(1+\alpha^{q^{j-1}}t^{2n-j+1})=A^{q^n+1}.$$
Since $b \in \mathbb{F}_q \setminus \{0\}$, Equation \eqref{eq:conditionsplit} can then be rewritten as
\begin{equation} \label{eq:claim}
\sum_{\ell=0}^{2n-1} (-1)^{\ell} \bigg( \frac{t \cdot A}{1+\alpha t}\bigg)^{q^\ell}=0.
\end{equation}
Observe that
$$A=t^{1+q+\ldots+q^{n-1}} \cdot (t^{-1}+\alpha) \cdot (t^{-q^{n-1}}+\alpha^q) \cdot \ldots \cdot (t^{-q}+\alpha),$$
where $t^{1+q+\ldots+q^{n-1}} \in \mathbb{F}_q \setminus \{0\}$, since $t \in \mathbb{F}_{q^n} \setminus \mathbb{F}_q$. Writing $1+\alpha t=(\alpha+t^{-1})t$ and denoting $s:=t^{-1}$, we see that Equation \eqref{eq:claim} is equivalent to
\begin{equation} \label{changevar}
\sum_{\ell=0}^{2n-1} (-1)^{\ell} \bigg( \prod_{j=0}^{n-2} (s^{q^{j+1}}+\alpha^{q^j}) \bigg)^{q^\ell}=0.
\end{equation}
Since the choice of $\alpha$ is arbitrary as long as the condition $\alpha \in \mathbb{F}_{q^2} \setminus \mathbb{F}_q$ is satisfied, we may replace $\alpha$ with $\alpha^q$, and correspondingly replace Equation \eqref{changevar} by the equivalent equation
\begin{equation} \label{rew}
\sum_{\ell=0}^{2n-1} (-1)^{\ell} \bigg( \prod_{j=0}^{n-2} (s^{q^{j}}+\alpha^{q^j}) \bigg)^{q^\ell}=\sum_{\ell=0}^{2n-1} (-1)^{\ell} (s+\alpha)^{q^\ell \cdot \frac{q^{n-1}-1}{q-1}}=0.
\end{equation}
for $\alpha \in \mathbb{F}_{q^2} \setminus \mathbb{F}_q$ and all $s \in \mathbb{F}_{q^n} \setminus \mathbb{F}_q$.
In other words, the theorem will follow if we can show that the polynomial
$$P(T):=\sum_{\ell=0}^{2n-1} (-1)^\ell \ T^{q^\ell \cdot \frac{q^{n-1}-1}{q-1}}$$
vanishes for all values of the form $T=s+\alpha$, where $\alpha \in \mathbb{F}_{q^2} \setminus \mathbb{F}_q$ and $s \in \mathbb{F}_{q^n} \setminus \mathbb{F}_q$.
Observe that for any such value of $T=s+\alpha$, we have $T^{q^{2n}}-T=0$ and $T^{q^n}-T-c=0$, where $c:=\alpha^q-\alpha \in \mathbb{F}_{q^2}$ satisfies $c^q+c=0$. Therefore, we may simplify the condition $P(T)=0$ by replacing $P(T)$ with any polynomial that is equivalent to it modulo the polynomials $T^{q^{2n}}-T$ and $T^{q^n}-T-c.$ Concretely, we will use this to replace $P(T)$ by a polynomial of lower degree.
First of all, for $\ell \geq n+1$ we have
\begin{align*}
q^\ell \frac{q^{n-1}-1}{q-1}& =q^\ell \frac{q^{2n-\ell}-1}{q-1} + q^{2n} \frac{q^{\ell-n-1}-1}{q-1}.
\end{align*}
Hence working modulo $T^{q^{2n}}-T$ the condition $P(T)=0$ can be simplified to
\begin{equation*} \label{ok}
\bar P(T)=\sum_{\ell=0}^{n} (-1)^\ell T^{q^\ell \frac{q^{n-1}-1}{q-1}}+\sum_{\ell=n+1}^{2n-1} (-1)^\ell T^{q^\ell \frac{q^{2n-\ell}-1}{q-1} + \frac{q^{\ell-n-1}-1}{q-1}}=0.
\end{equation*}
Similarly working modulo $T^{q^n}-T-c$ and using that $(-1)^n=-1$, the condition $\bar P(T)=0$ can be simplified to
\begin{align}\label{eq:phat}
0&=T^{\frac{q^{n-1}-1}{q-1}} + \sum_{\ell=1}^{n} (-1)^\ell T^{\frac{q^n-q^\ell}{q-1}} (T+c)^{\frac{q^{\ell-1}-1}{q-1}} + \sum_{\ell=n+1}^{2n-1} (-1)^\ell T^{\frac{q^{\ell-n-1}-1}{q-1}} (T+c)^{\frac{q^{n}-q^{\ell-n}}{q-1}} \notag\\
& = \sum_{\ell=1}^{n} (-1)^\ell T^{\frac{q^n-q^\ell}{q-1}} (T+c)^{\frac{q^{\ell-1}-1}{q-1}}- \sum_{\ell=1}^{n} (-1)^\ell (T+c)^{\frac{q^n-q^\ell}{q-1}} T^{\frac{q^{\ell-1}-1}{q-1}}.
\end{align}
Now let us write
$$P_1(T)=\sum_{\ell=1}^{n} (-1)^\ell T^{\frac{q^n-q^\ell}{q-1}} (T+c)^{\frac{q^{\ell-1}-1}{q-1}} \ \makebox{and} \ P_2(T)=\sum_{\ell=1}^{n} (-1)^\ell (T+c)^{\frac{q^n-q^\ell}{q-1}} T^{\frac{q^{\ell-1}-1}{q-1}},$$
then what we need to show is that $P_1(T)=P_2(T)$ for any $T=s+\alpha$, where $\alpha^q-\alpha=c$ and $s \in \mathbb{F}_{q^n} \setminus \mathbb{F}_q.$ What we in fact will show is the stronger statement that $P_1(T)=P_2(T)$ as polynomials in $T$ for any $c \in \mathbb{F}_{q^2}$ such that $c^q+c=0$. The original claim about the splitting of places in the extension $\K(x,y)/\K(\rho)$ (and hence the theorem) will then follow as well.

At this point note that
\begin{align*}
(T+c)^{\frac{q^{\ell-1}-1}{q-1}}& =(T+c)^{1+q+\cdots+q^{\ell-2}}=\prod_{j \in \{0,\dots,\ell-2\}}(T^{q^j}+c^{q^j})\\
&=\sum_{I\subset\{0,\dots,\ell-2\}}\left(\prod_{j \in I}T^{q^{j}}\right)\left(\prod_{k \in \{0,\dots,\ell-2\}\setminus I}c^{q^{k}}\right),
\end{align*}
where the summation is over all subsets $I$ of $\{0,\dots,\ell-2\}$. Hence
$$P_1(T)=\sum_{\ell=1}^{n} (-1)^\ell \sum_{I\subset\{0,\dots,\ell-2\}}T^{q^{n-1}+\cdots+q^\ell}\left(\prod_{j \in I}T^{q^{j}}\right)\left(\prod_{k \in \{0,\dots,\ell-2\}\setminus I}c^{q^{k}}\right).$$
Similarly,
$$P_2(T)=\sum_{\ell'=1}^{n} (-1)^{\ell'} \sum_{J\subset\{\ell',\dots,n\}}\left(\prod_{j \in J}T^{q^{j}}\right)\left(\prod_{k \in \{\ell',\dots,n\}\setminus J}c^{q^{k}}\right)T^{q^{\ell'-2}+\cdots+1}.$$

Given $\ell$ and $I\subset\{0,\dots,\ell-2\}$, define $\ell' \ge 1$ to be the unique maximal integer such that $\{0,\dots,\ell'-2\} \subset I.$ The case $\ell'=1$ corresponds to the case that $0 \not\in I$. Then $\ell' \le \ell$, the  case $\ell'=\ell$ corresponding to the case that $I=\{0,\dots,\ell-2\}.$ Now, still for a given $\ell$ and $I\subset\{0,\dots,\ell-2\}$, define $J=\{\ell,\dots,n-1\}\cup (I \setminus \{0,\dots,\ell'-2\}).$ Note that $J \subset \{\ell',\dots,n\},$ since by definition of $\ell'$, the integer $\ell'-1$ is not in $I$. Then we have \begin{align*}
& (-1)^\ell T^{q^{n-1}+\cdots+q^\ell}\left(\prod_{j \in I}T^{q^{j}}\right)\left(\prod_{k \in \{0,\dots,\ell-2\}\setminus I}c^{q^{k}}\right) =\\
\\&(-1)^\ell\dfrac{c^{q^{\ell'-1}}}{c^{q^{\ell-1}}}\left(\prod_{j \in J}T^{q^{j}}\right)\left(\prod_{k \in \{\ell',\dots,n\}\setminus J}c^{q^{k}}\right)T^{q^{\ell'-2}+\cdots+1}=\\
\\&(-1)^{\ell'} \left(\prod_{j \in J}T^{q^{j}}\right)\left(\prod_{k \in \{\ell',\dots,n\}\setminus J}c^{q^{k}}\right)T^{q^{\ell'-2}+\cdots+1},
\end{align*}
where in the last equality we used that $c^{q^{\ell'-1}}=(-1)^{\ell-\ell'}c^{q^{\ell-1}}$, since $c^q=-c.$ Summing this equality over $1 \le \ell \le n$ and all subsets $I \subset \{0,\dots,\ell-1\},$ we obtain that $P_1(T)=P_2(T)$ as claimed.
\end{proof}

\section{Determination of the full automorphism group of $\mathcal{X}_n$, $n \geq 5$}

We will now determine the complete automorphism group of $\mathcal X_n$ for $n \ge 5.$ We already observed in Theorem \ref{lift} that $Aut(\cX_n)$ contains a subgroup $G$, given in Equation \eqref{eq:groupG}, isomorphic to $SL(2,q) \rtimes C_{q^n+1}$. The group $G$ is obtained lifting in $m$ ways the maximal subgroup $M_\ell \cong SL(2,q) \rtimes C_{q+1}$ of $Aut(\cH_q) \cong PGU(3,q)$ fixing the non-tangent line $\ell:X+Y=0$. 
The aim of this section is to prove that for $n \ge 5$, $G$ is the full automorphism group of $\cX_n$, implying also that $M_\ell$ is the largest subgroup $Aut(\cH_q)$ that can be lifted to $Aut(\cX_n)$, i.e, that $PGU(3,q)$ can be lifted if and only if $n=3$.
The following lemma describes the strategy that will be used to show that $G$ is the full automorphism group of $\cX_n$. The affine equations in \eqref{xn} give rise to an affine part of the curve $\cX_n$. Considering its projective closure, $q+1$ points of $\cX_n$ at infinity appear. The places of $\K(x,y,z)$ corresponding to these will be denoted by $R_\infty^1,\dots,R_\infty^{q+1}$ and will be called the infinite places of $\K(x,y,z).$ Moreover, we denote their restrictions to $\K(x,y)$ by $P_\infty^1,\dots,P_\infty^{q+1}$, consistent with the previous notation. Finally recall the we assumed that place $P_\infty^1$ is centered around the point $(1:-1:0)$, implying that $R_\infty^1$ is centered around the infinite point $(1:-1:0:0)$ of $\cX_n.$

\begin{lemma} \label{normalizer}
Let $C_m= Gal(\K(x,y,z) / \K(x,y))$. If $Aut(\cX_n)$ acts on the set $\Omega=\{R_\infty^1 , \ldots, R_\infty^{q+1}\}$ of the $q+1$ infinite places of $\K(x,y,z)$, then $Aut(\cX_n)=G.$
\end{lemma}

\begin{proof}
From the proof of Lemma \ref{genus}, the place $P_\infty^i$ is totally ramified in the extension $\K(x,y,z)/\K(x,y)$. Moreover, by definition, the place $R_\infty^i$ lies above it. Hence $C_m$ fixes each of the places $R_\infty^1,\dots,R_\infty^{q+1}$. Let
$$T:=\{\alpha \in Aut(\cX_n) : \  \alpha(P)=P, \ \rm{for \ every} \ P \in \Omega\}.$$
Clearly, $C_m$ is a subgroup of $T$.

We claim that $T$ is a normal subgroup of $Aut(\cX_n)$. Indeed, let $\alpha \in T$, $\beta \in Aut(\cX_n)$ and $R_\infty^i \in \Omega$. Since $Aut(\cX_n)$ acts on $\Omega$, we have $\beta(R_\infty^i)=R_\infty^j$ for some $j=1,\ldots,q+1$. Thus, $\beta^{-1} \alpha \beta(R_\infty^i)=\beta^{-1} \alpha(R_\infty^j)=\beta^{-1}(R_\infty^j)=R_\infty^i$. Hence $\beta^{-1} \alpha \beta$ fixes $\Omega$ pointwise for every  $\alpha \in T$ and $\beta \in Aut(\cX_n)$. This proves the claim.

We now claim that the characteristic $p$ does not divide $|T|$. If it did, then $T$ would contain an element of order $p$, which by definition would fix the $q+1$ places in $\Omega$. On the other hand, Theorem \ref{thm:maximal} implies that $\cX_n$ has zero $p$-rank. Then \cite[Lemma 11.129]{HKT} implies that any automorphism of order $p$ cannot have more than one fixed place. This proves the second claim.

Thus, $T$ is a normal subgroup of $Aut(\cX_n)$ of order coprime with $p$, whose elements fix more than one place. This implies that $T$ is cyclic using \cite[Theorem 11.49]{HKT}. Let $\beta \in Aut(\cX_n)$ be chosen arbitrary and consider the conjugate $\beta^{-1} C_m \beta=C_m^{\beta}$. Then ${C}_m^{\beta}$ is a cyclic subgroup of $T$ of order $m$. Since $T$ is cyclic and $C_m$ is contained in $T$, we have that ${C}_m^\beta=C_m$. This shows that ${C}_m^\beta=C_m$ for every $\beta \in Aut(\cX_n)$ and hence that $Aut(\cX_n)=N_{Aut(\cX_n)}(C_m)$.

Consider the quotient group $\tilde G=Aut(\cX_n) / C_m$. Then $\tilde G$ is a subgroup of $Aut(\cH_q) \cong PGU(3,q)$ acting on the $q+1$ infinite places $P_\infty^1,\ldots,P_\infty^{q+1}$ of $\K(x,y)$ with $R_\infty^i|P_\infty^i$ for every $i=1,\ldots,q+1$. This implies that $\tilde G$ is contained in the stabilizer $M_\ell$ of the non-tangent line $\ell$ and hence $|\tilde G|$ divides $q(q-1)(q+1)^2$, see \cite[Theorem A.10]{HKT}. Note that $G/C_m \leq \tilde G$ and hence $|\tilde G| \geq |G|/m=q(q-1)(q+1)^2$. Hence $|\tilde G|=q(q-1)(q+1)^2$ and $\tilde G=G/C_m$. Moreover, it shows that $|Aut(\cX_n)|=|\tilde G|\cdot |C_m|=|G|$ and hence that $Aut(\cX_n)=G$.
\end{proof}

The aim of the rest of this section is to prove that $Aut(\cX_n)$ acts on the set $\Omega=\{R_{\infty}^1,\ldots,R_{\infty}^{q+1}\}$, since then the previous lemma implies that $Aut(\cX_n)=G$. We will consider Weierstrass semigroups $H(R)$ of various $\K$-rational places $R$ of $\K(x,y,z)$ in order to achieve this. Since $G$ contains a group acting transitively on $\Omega$, it is sufficient to show that $H(R_\infty^1) \ne H(R)$ for every $i=1,\ldots,q+1$ and $R \in \cX_n(\mathbb{F}_{q^{2n}}) \setminus \Omega,$ where, as before, $R_\infty^1$ denotes the place centered at the infinite point $(1:-1:0:0)$. We will use also the notation $G(R)$ for the set of gaps of a place $R$, i.e., $G(R):=\mathbb{N} \setminus H(R).$

\begin{lemma}
Let $P \in \cX_n(\mathbb{F}_{q^{2n}}) \setminus \Omega$. Then $mq^2+mq+q^2-q \in G(P)$ while $mq^2+mq+q^2-q \in H(P_{\infty}^i)$ for every $i=1,\ldots,q+1$. In particular $H(P) \ne H(P_\infty^i)$ for every $i=1,\ldots,q+1$ and $Aut(\cX_n)$ acts on $\Omega$.
\end{lemma}

\begin{proof}
To prove that $mq^2+mq+q^2-q \in H(R_{\infty}^i)$ for every $i=1,\ldots,q+1$, it is sufficient to show that $mq, mq+q^2-q \in H(R_\infty^1)$. Let $P_\infty^1$ denote the restriction of $R_\infty^1$ to $\K(x,y)$.
Equation \eqref{div:rho} implies that
$$(\rho)_{\K(x,y,z)}=qm R_{\infty}^{1} - m \sum_{i=2}^{q+1} R_{\infty}^i,$$
so that $1/\rho$ has a pole of order $mq$ at $R_\infty^1$ and no poles elsewhere. Hence $mq \in H(R_\infty^1).$
Moreover,
$$(z)^{(\infty)}_{\K(x,y,z)} =  (q^2-q)  \sum_{i=1}^{q+1} R_{\infty}^i,$$
where $(z)^{(\infty)}_{\K(x,y,z)}$ denotes the pole divisor of $z$.
Hence,
$$\bigg(\frac{z}{\rho}\bigg)^{(\infty)}_{\K(x,y,z)}=(qm+q^2-q)R_{\infty}^1.$$
Here we used that $m-q^2+q>0$ as $n \geq 3$. Hence $mq+q^2-q \in H(R_\infty^1)$.

We now prove that $mq^2+mq+q^2-q \in G(R)$ for every $R \in \cX_n(\mathbb{F}_{q^{2n}}) \setminus \Omega$. To do that it is sufficient to show that there exists a regular (sometimes also called holomorphic) differential $\omega$ such that $v_R(\omega)=mq^2+mq+q^2-q-1$, see \cite[Corollary 14.2.5]{VS}. In the extension $\K(x,y,z) / \K(z)$ the unique ramified place is the pole of $z$ above which lie the places in $\Omega$. Moreover the group $G$ contains a subgroup of order $q+1$ acting sharply transitively on $\Omega$, while fixing $z$ and $dz$, so $\mathrm{supp}(dz)=\Omega$. Since $\deg(dz)=2g(\cX_n)-2$ and $|\Omega|=q+1$, we get that $$(dz)_{\K(x,y,z)}=\frac{2g(\cX_n)-2}{q+1}(R_\infty^1+\ldots+R_\infty^{q+1})
=(q^{n+1}-q^n-q^2+2q-2)(R_\infty^1+\ldots+R_\infty^{q+1}).$$
Hence the differential $fdz$ is regular if and only if $f \in L((q^{n+1}-q^n-q^2+2q-2)(R_\infty^1+\ldots+R_\infty^{q+1}))$.
In order to prove that $mq^2+mq+q^2-q \in G(R)$ it is therefore sufficient to construct a function $f \in L((q^{n+1}-q^n-q^2+2q-2)(R_\infty^1+\ldots+R_\infty^{q+1}))$ such that $v_R(f)=mq^2+mq+q^2-q-1$.

From Theorem \ref{thm:maximal} $\cX_n$ is $\mathbb{F}_{q^{2n}}$-maximal and hence from the Fundamental Equation \cite[Theorem 9.79]{HKT}, there exists a function $t_R$ such that $(t_R)_{\K(x,y,z)}=(q^n+1)(R-R_\infty^1)$. In particular,
$$(t_R \cdot \rho)_{\K(x,y,z)}=(q^n+1)(R-R_\infty^1)+qm R_{\infty}^{1} - m \sum_{i=2}^{q+1} R_{\infty}^i=(q^n+1)R-m \sum_{i=1}^{q+1} R_{\infty}^i.$$
Moreover, since $R \in \cX_n(\mathbb{F}_{q^{2n}}) \setminus \Omega$, $R$ is centered at a certain affine point, say $(a,b,c)$ with $a,b,c \in \mathbb{F}_{q^{2n}}$. Observe that $(z-c)\ge -(q^2-q)(R_\infty^1+\ldots+R_\infty^{q+1})$ and since in the extension $\K(x,y,z) / \K(z)$ only the pole of $z$ is ramified, we have $v_R(z-c)=1$.
Let $f=(t_R \cdot \rho)^q (z-c)^{q^2-q-1}$, then $$(f)_{\K(x,y,z)} \ge -(mq+(q^2-q)(q^2-q-1))(R_\infty^1+\ldots+R_\infty^{q+1}).$$
Observing that $mq+(q^2-q)(q^2-q-1)<q^{n+1}-q^n-q^2+2q-2$, since $n \ge 5$, we get that $f \in L((q^{n+1}-q^n-q^2+2q-2)(R_\infty^1+\ldots+R_\infty^{q+1}))$ and hence that $fdz$ is a regular differential.
Moreover,
$$v_R(fdz)=v_R(fd(z-c))=q(q^n+1)+q^2-q-1=(q^2+q)m+q^2-q-1.$$ Considering the differential $w=fdz$
we obtain that $mq^2+mq+q^2-q \in G(R)$.
\end{proof}

Combining the previous results the full automorphism group of $\cX_n$ has now been determined.

\begin{theorem}
Let $q$ be a prime power and $n \geq 5$ odd. The full automorphism group of $\cX_n$ is $Aut(\cX_n)=G \cong SL(2,q) \rtimes C_{q^n+1}$. Moreover, $Aut(\cX_n)$ has a non-tame short orbit $\Omega$ consisting of the $q+1$ infinite places of $\K(x,y,z)$.
\end{theorem}

\begin{remark}
$Aut(\cX_n)$ is asymptotically large, meaning that for every $q \geq 43$ and $n \geq 5$, the Classical Hurwitz upper bound $84(g(\cX_n)-1)$ does not hold.
\end{remark}

\begin{remark}
We have seen that $\mathcal C_3 \cong \mathcal X_3$ and that for any $q$ and any odd $n \ge 5$, the curves $\mathcal C_n$ and $\mathcal X_n$ have the same genus. One could suspect that if $n \ge 5$, the curves $\mathcal C_n$ and $\mathcal X_n$, though not isomorphic, are twists of each other, i.e., are isomorphic over $\overline{\mathbb{F}}_q$. If this would be the case, their automorphism groups over $\overline{\mathbb{F}}_q$ would be identical. However, it was shown in \cite[Thm. 3.10]{GSY} that if $\mathcal K$ is a maximal curve over $\mathbb{F}_{q^2}$ of genus at least $2$, then any automorphism of $\overline{\mathcal K}$, the curve obtained from $\mathcal K$ by extending the field of definition to $\overline{\mathbb{F}}_q$, is already defined over $\mathbb{F}_{q^2}$. Since over $\K$ and $n \ge 5$, the curves $\mathcal X_n$ and $\mathcal C_n$ have distinct automorphism groups, this result implies that they cannot be twists of each other.
\end{remark}

Though the previous results show that the curves $\cX_n$ for $n \geq 5$ exhibit different and new properties than previously known maximal curves, they do not provides examples of maximal curves with a previously unknown genus. In other words, these curves themselves do not add new values to the genus spectrum of maximal curves. Nevertheless, it is possible to provide new genera for maximal curves considering subfields of $\mathbb{F}_{q^{2n}}(x,y,z)$. As an example, in the following remark we consider some specific subfields of $\mathbb{F}_{q^{2n}}(x,y,z)$ with $n \geq 5$ which in some cases will produce (up to our knowledge) previously unknown values to this genus spectrum.

\begin{remark} \label{genera}
Let $q$ be a prime power and $n \geq 5$ odd. Let $k_1$ and $k_2$ be divisors of $q+1$ and $m=(q^n+1)/(q+1)$ respectively. The curve $\cY_{k_1,k_2}$ with function field $\mathbb{F}_{q^{2n}}(w,s)$ defined by
\begin{equation} \label{subcover}
w^{\frac{q^n+1}{k_2}}=s^{\frac{q+1}{k_1}} \cdot \dfrac{(s^{\frac{q^2-1}{k_1}}-1)^{q+1}}{(s^{\frac{q+1}{k_1}}-1)^q},
\end{equation}
is $\mathbb{F}_{q^{2n}}$-maximal of genus
\begin{equation} \label{eqyk1k2}
g(\cY_{k_1,k_2})=1+\frac{1}{2} \bigg[-\delta_1-\frac{q+1}{k_1}+\frac{(q^2-1)(q^n+1)}{k_1k_2}-\frac{q^2-1}{k_1} \delta_2 + \frac{q+1}{k_1} \delta_2-\delta_3\bigg],
\end{equation}
where $k=(q^n+1)/k_2$, $\delta_1=\gcd(k,\frac{q+1}{k_1})$, $\delta_2=\gcd(k,q+1)$ and $\delta_3=\gcd(k,\frac{(q^2-1)}{k_1})$.
Moreover, $\cY_{k_1,k_2}$ is a Galois subcover of $\cX_n$ of degree $k_1 \cdot k_2$.
\end{remark}

\begin{proof}
Let $H$ be the subgroup of the automorphism group of $\mathbb{F}_{q^{2n}}(x,y,z)$ generated by $\gamma_{k_1}$ and $\gamma_{k_2},$ where
$$\gamma_{k_1}: (x,y,z) \mapsto (\lambda x , \lambda^{-1} y, z), \quad \gamma_{k_2}: (x,y,z) \mapsto (x , y, \xi z), $$
and $\lambda,\xi \in \mathbb{F}_{q^{2n}}$ with $\ord(\lambda)=k_1$ and $\ord(\xi)=k_2$. Then $H$ is abelian of order $k_1 \cdot k_2$.
More precisely, $H = C_{k_1} \times C_{k_2}$, where $ C_{k_1}=\langle \gamma_{k_1} \rangle$ and $C_{k_2}=\langle \gamma_{k_2} \rangle$.
Let $s=x^{k_1}$ and $w=z^{k_2}$. Then $w$ and $s$ are fixed by $\mathcal{G}$ and
$$w^{\frac{q^n+1}{k_2}}=z^{q^n+1}=\frac{(x^{q^2}-x)^{q+1}}{(x^{q+1}-1)^q}=s^{{\frac{q+1}{k_1}}} \cdot \frac{(s^{\frac{q^2-1}{k_1}}-1)^{q+1}}{(s^{\frac{q+1}{k_1}}-1)^q}.$$
This proves that $\mathbb{F}_{q^{2n}}(s,w) \subseteq Fix(H)$ and $w^{\frac{q^n+1}{k_2}}=s^{{\frac{q+1}{k_1}}} \cdot (s^{\frac{q^2-1}{k_1}}-1)^{q+1}/(s^{\frac{q+1}{k_1}}-1)^q$.

To show that the curve defined by Equation \eqref{subcover} is irreducible and to compute the genus of $\mathbb{F}_{q^{2n}}(s,w)$ we note that
\begin{align*}
\Bigg( s^{{\frac{q+1}{k_1}}} \cdot &\frac{(s^{\frac{q^2-1}{k_1}}-1)^{q+1}}{(s^{\frac{q+1}{k_1}}-1)^q}\Bigg)_{\mathbb{F}_{q^{2n}}(s)}=\\
\\& \frac{(q+1)}{k_1}P_{s=0}+\sum_{\alpha^{\frac{q+1}{k_1}}=1} P_{s=\alpha}+(q+1)\cdot \sum_{\beta^{\frac{q^2-1}{k_1}}=1, \  \beta^\frac{q+1}{k_1} \ne 1} P_{s=\beta}-\frac{q(q^2-1)}{k_1}P_{s=\infty},
\end{align*}
and hence $\mathbb{F}_{q^{2n}}(s,w)/\mathbb{F}_{q^{2n}}(s)$ is a Kummer extension of degree $k=(q^n+1)/k_2$.
From \cite[Corollary 3.7.4]{Sti},
\begin{align*}
g(\mathbb{F}_{q^{2n}}(s,w)) & =1-k+\frac{1}{2} \bigg[k-\delta_1+\frac{q+1}{k_1} (k-1) +\bigg(\frac{q^2-1}{k_1}-\frac{q+1}{k_1} \bigg)(k-\delta_2)+k-\delta_3\bigg]\\
&=1+\frac{1}{2} \bigg[-\delta_1-\frac{q+1}{k_1}+\frac{(q^2-1)(q^n+1)}{k_1k_2}-\frac{q^2-1}{k_1} \delta_2 + \frac{q+1}{k_1} \delta_2-\delta_3\bigg],
\end{align*}
where $\delta_1=\gcd(k,\frac{q+1}{k_1})$, $\delta_2=\gcd(k,q+1)$ and $\delta_3=\gcd(k,\frac{(q^2-1)}{k_1})$.

In order to prove that $\mathbb{F}_{q^{2n}}(s,w) = Fix(H)$ consider the extension $\mathbb{F}_{q^{2n}}(x,y,z) / \mathbb{F}_{q^{2n}}(z^{q^n+1}) $. Recall that from Theorem \ref{lift}, $[\mathbb{F}_{q^{2n}}(x,y,z) : \mathbb{F}_{q^{2n}}(z^{q^n+1})]=|G|=(q^n+1)q(q^2-1)$. Hence,
\begin{align*}
&(q^n+1)q(q^2-1)=\\
&[\mathbb{F}_{q^{2n}}(x,y,z) : \mathbb{F}_{q^{2n}}(s,w) ] \cdot [\mathbb{F}_{q^{2n}}(s,w)  : \mathbb{F}_{q^{2n}}(s,w^{\frac{q^n+1}{k_2}}) ] \cdot [ \mathbb{F}_{q^{2n}}(s,w^{\frac{q^n+1}{k_2}})  : \mathbb{F}_{q^{2n}}(w^{\frac{q^n+1}{k_2}}) ]=\\
\\&[\mathbb{F}_{q^{2n}}(x,y,z) : \mathbb{F}_{q^{2n}}(s,w) ] \cdot \frac{q^n+1}{k_2} \cdot \frac{q(q^2-1)}{k_1},
\end{align*}
where the equality $ [ \mathbb{F}_{q^{2n}}(s,w^{\frac{q^n+1}{k_2}})  : \mathbb{F}_{q^{2n}}(w^{\frac{q^n+1}{k_2}}) ]=\frac{q(q^2-1)}{k_1}$ follows from Equation \eqref{subcover} observing that the degree in $s$ equals $\frac{q+1}{k_1}+\frac{(q+1)(q^2-1)}{k_1}-\frac{q(q+1)}{k_1}=\frac{q(q^2-1)}{k_1}$.
Thus, $[\mathbb{F}_{q^{2n}}(x,y,z) : \mathbb{F}_{q^{2n}}(s,w) ]=k_1 \cdot k_2 = |H|$ and $\mathbb{F}_{q^{2n}}(s,w)=Fix(H)$.
\end{proof}

Remark \ref{genera} provides new equations of $\mathbb{F}_{q^{2n}}$-maximal curves for every $q$ and $n \ge 5$ odd. Examples for small values of $q$ and $n$ are listed below. We have indicated a genus with boldface if it gives a new addition to the genus spectrum of $\mathbb{F}_{q^{2n}}$-maximal curves. To check whether or not a genus gave a new addition, we compared our results with the genera of maximal function fields given in \cite{AQ,ABB,BMXY,DO,FG,GOS,MX,MZ}.
\begin{itemize}
\item $q=4$ and $n=5$: $\{32, 156, {\bf 302}, 1506, {\bf 1532} \}$,
\item $q=4$ and $n=7$: $\{{\bf 212}, {\bf 842}, 1056, 4206, {\bf 24572}\}$,
\item $q=5$ and $n=5$: $\{{\bf 6242}, {\bf 12484}, 18724\}$,
\item $q=7$ and $n=5$: $\{{\bf 243}, {\bf 485}, 969, {\bf 1941}, {\bf 4563 }, {\bf 9125}, 18249, 36501, {\bf 50403 }, {\bf 100805}, 201609 \}$.
\end{itemize}

\section*{Acknowledgments}

The second author would like to thank the Italian Ministry MIUR, Strutture Geometriche, Combinatoria e loro Applicazioni, Prin 2012 prot.~2012XZE22K and GNSAGA of the Italian INDAM.

\vspace{1ex}
\noindent
Peter Beelen

\vspace{.5ex}
\noindent
Technical University of Denmark,\\
Department of Applied Mathematics and Computer Science,\\
Matematiktorvet 303B,\\
2800 Kgs. Lyngby,\\
Denmark,\\
pabe@dtu.dk\\

\vspace{1ex}
\noindent
Maria Montanucci

\vspace{.5ex}
\noindent
Universit\'a degli Studi della Basilicata,\\
Dipartimento di Matematica, Informatica ed Economia,\\
Campus di Macchia Romana,\\
Viale dell' Ateneo Lucano 10,\\
85100 Potenza,\\
Italy,\\
maria.montanucci@unibas.it

\end{document}